\documentclass[12pt,a4paper]{amsart}
\usepackage{amsfonts,amssymb,amsmath,amsthm}
\usepackage{enumerate}
\usepackage{bbm}
\usepackage{amssymb}
\usepackage{mathrsfs}
\usepackage[usenames]{xcolor}
\usepackage[left=2.5cm, top=2.5cm,bottom=3cm,right=2.5cm]{geometry}
\usepackage{graphicx}

\newtheorem{theorem}{Theorem}
\newtheorem{lem}[theorem]{Lemma}

\theoremstyle{definition}
\newtheorem{definition}[theorem]{Definition}
\newtheorem{remark}[theorem]{Remark}
\numberwithin{equation}{section}

\DeclareMathOperator{\disp}{disp}

\newcommand{\pointset}[1]{\ensuremath{{\mathcal P}_{#1}}}
\newcommand{\fiblat}[1]{\ensuremath{{\mathcal F}_{#1}}}
\newcommand{\N}{\ensuremath{{\mathbb N}}}

\newcommand{\eps}{\varepsilon}

\newcommand{\dint}{\,{\rm d}}

\author{ Simon Breneis \and Aicke Hinrichs}
\date{\today}
\address[Simon Breneis]{Institut f\"ur Analysis\\
Johannes Kepler Universit\"at Linz\\
Altenbergerstrasse 69\\
4040 Linz\\
Austria}
\email{simon.breneis@jku.at}
\address[Aicke Hinrichs]{Institut f\"ur Analysis\\
Johannes Kepler Universit\"at Linz\\
Altenbergerstrasse 69\\
4040 Linz\\
Austria}
\email{aicke.hinrichs@jku.at}

\thanks{Both authors are supported by the Austrian Science Fund (FWF) Project F5513-N26, which is a part of the Special Research Program ``Quasi-Monte Carlo Methods: Theory and Applications''.
AH would like to thank the Isaac Newton Institute for 
Mathematical Sciences for support and hospitality during the programme
``Approximation, sampling and compression in data science''
when work on this paper was undertaken. 
This work was supported by EPSRC Grant Number EP/R014604/1.
AH was also supported by a grant from the
Simons Foundation.}

\begin{document}

\title[Fibonacci lattices have minimal dispersion]{Fibonacci lattices have minimal dispersion on the two-dimensional torus}

\begin{abstract}
We study the size of the largest rectangle containing no point of a given point set in the two-dimensional torus, the dispersion of the point set. A known lower bound for the dispersion of any point set of cardinality $n\ge 2$ in this setting is $2/n$. We show that if $n$ is a Fibonacci number then the Fibonacci lattice has dispersion exactly $2/n$ meeting the lower bound. Moreover, we completely characterize integration lattices achieving the lower bound and provide insight into the structure of other optimal sets. We also treat related results in the nonperiodic setting.   
\end{abstract}

\maketitle

% % % % % % % % % % % % % % % % % % % % % %
\section{Introduction and main result}
% % % % % % % % % % % % % % % % % % % % % % 

We identify the two-dimensional torus with $[0,1]^2$. 
Any two points $x,y \in [0,1]^2$ define a rectangle $B(x,y)$ in the two-dimensional torus.
If $x=(x_1,x_2), y=(y_1,y_2)$ satisfy $x_1 \le y_1$ and $x_2 \le y_2$, this is the ordinary rectangle $B(x,y) = [x_1,y_1] \times [x_2,y_2]$. If $x_1 > y_1$ and $x_2 \le y_2$, then $B(x,y)= \big([0,y_1] \cup [x_1,1]\big) \times [x_2,y_2]$ is wrapped around in the direction of the first coordinate axis. Analogously, for $x_1 \le y_1$ and $x_2 > y_2$, it is wrapped around the direction of the second coordinate axis, and for $x_1 > y_1$ and $x_2 > y_2$ around both axis, see Figure \ref{fig:perbox}.

\begin{figure}[htb]\label{fig:perbox}
\caption{Periodic rectangles}
    \centering
    \begin{minipage}[t]{0.25\linewidth}
        \includegraphics[width=0.95\linewidth]{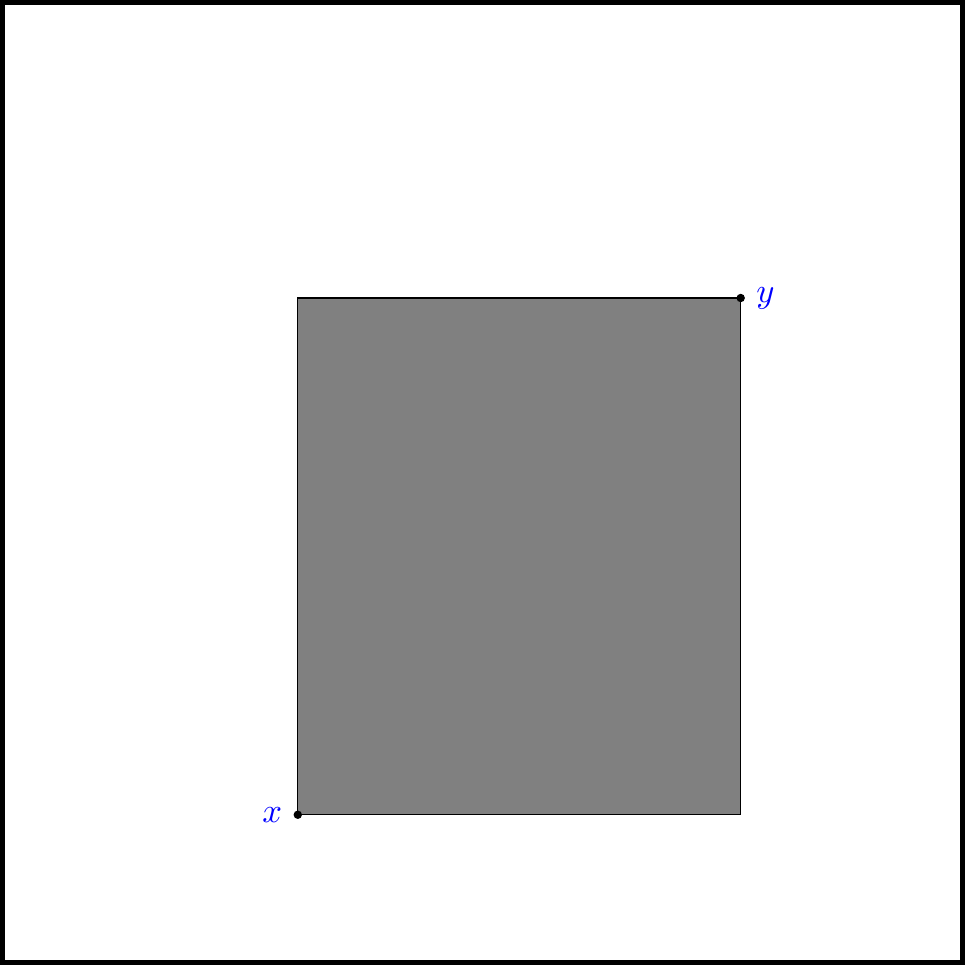}        
    \end{minipage}% 
    \hfill
    \begin{minipage}[t]{0.25\linewidth}
        \includegraphics[width=0.95\linewidth]{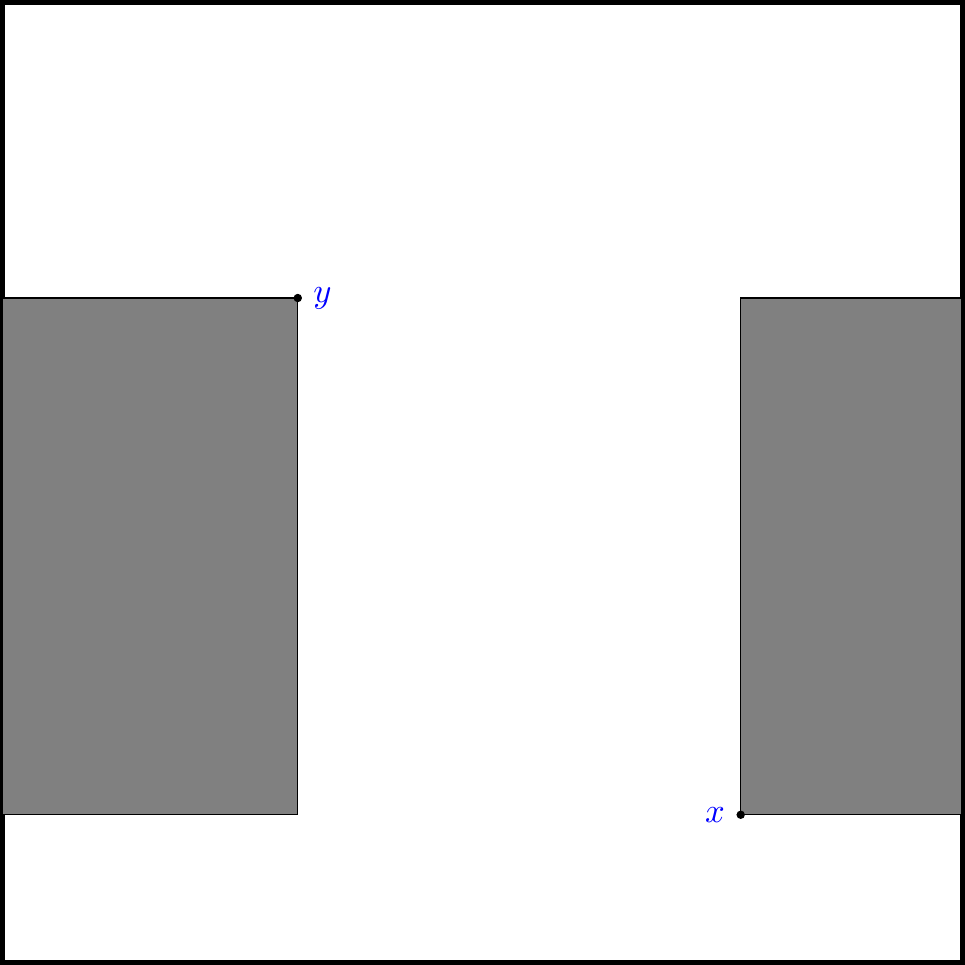}
    \end{minipage}% 
    \hfill
    \begin{minipage}[t]{0.25\linewidth}
        \includegraphics[width=0.95\linewidth]{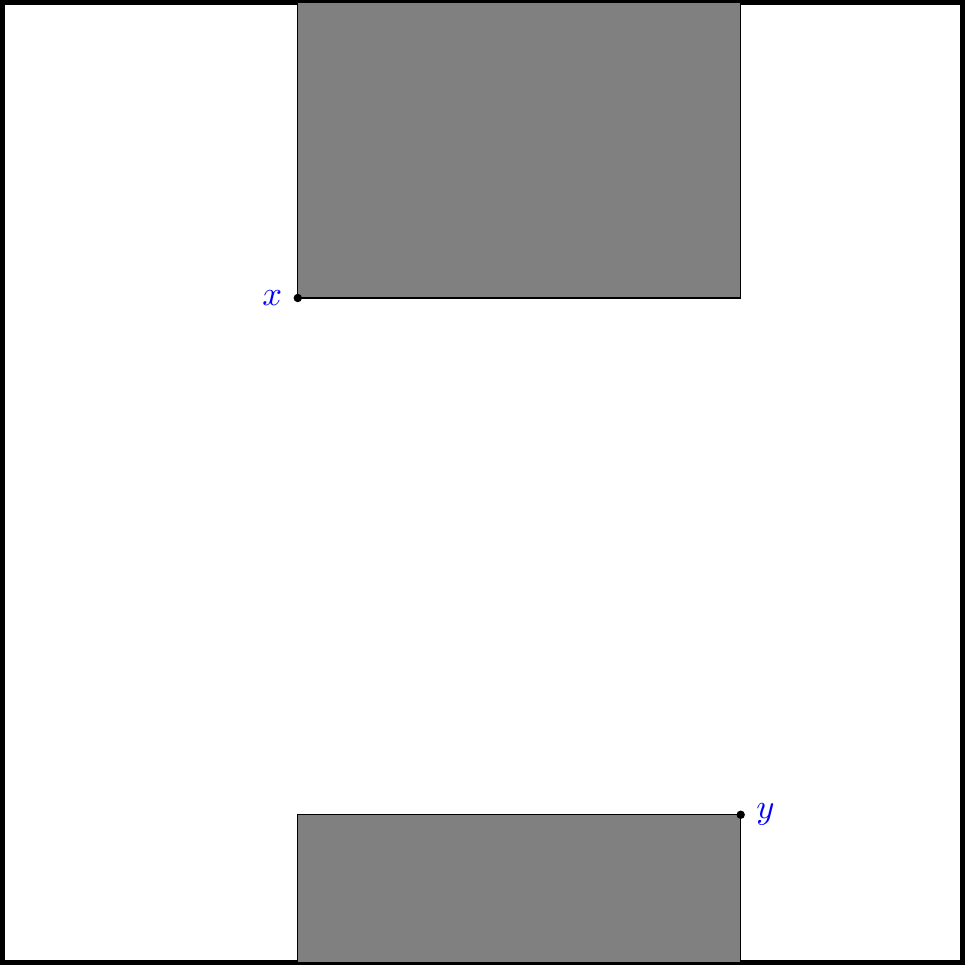} 
    \end{minipage}% 
    \hfill
    \begin{minipage}[t]{0.25\linewidth}
        \includegraphics[width=0.95\linewidth]{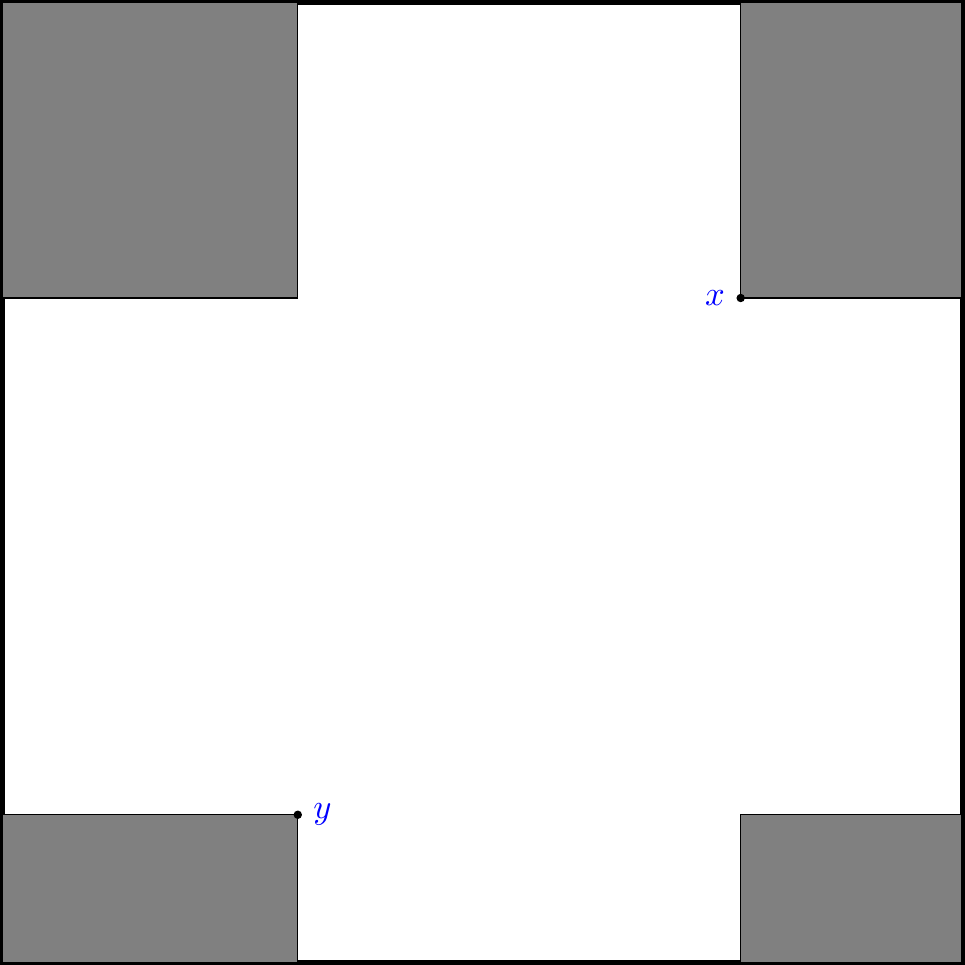}
    \end{minipage}
\end{figure}

For a given finite point set $\pointset{} \subset [0,1]^2$, the dispersion $\disp(\pointset{})$ of  $\pointset{}$ is the area of the largest rectangle $B(x,y)$ containing no point of $\pointset{}$ in the interior. The following lower bound follows as a special case from the result of M. Ullrich in \cite{U2018} for the $d$-dimensional torus in the case $d=2$.

\begin{theorem}\label{thm:lowerbound}
 For any $n\in \N$ with $n\ge 2$ and any point set $\pointset{n} \subset [0,1]^2$ 
 with $ \#\pointset{n} = n$, we have $$ \disp(\pointset{n}) \ge \frac{2}{n}. $$
\end{theorem}

The main purpose of this note is the investigation whether, and if so, for which sets, this bound is sharp.
  
It is well understood that the Fibonacci lattice has exceptional uniform distribution properties. We shortly discuss some results in this direction for the discrepancy and the dispersion as measures of uniform distribution. Let $(F_m)_{m\in\N}$ be the sequence of Fibonacci numbers starting with $F_1=F_2=1$ and defined via the recursive relation $F_{m+2}=F_m+F_{m+1}$ for $m\ge 2$. The Fibonacci lattice $\fiblat{m}$ is defined as 
$$\fiblat{m} := \Big\{\Big(\frac{k}{F_m},\Big\{\frac{kF_{m-2}}{F_m}\Big\}\Big): k\in\{0,1,...,F_m-1\}\Big\}.$$
Here, $\{\alpha\}$ denotes the fractional part of $\alpha$.

The Fibonacci lattice is an example of an integration lattice. A general integration lattice in dimension $d=2$ has the form
$$\Big\{\Big(\frac{k}{n},\Big\{\frac{k q}{n}\Big\}\Big): k\in\{0,1,...,n-1\}\Big\}.$$ Here $n$ and $1\leq q < n$ are positive integers. The number $q$ is called the generator of the integration lattice. It is sometimes required to be coprime to $n$. We do not make this additional requirement here. Observe that an integration lattice consists of $n$ points in $[0,1)^2$. For the theory of integration lattices and applications to numerical integration we refer to \cite{SJ1994}.  

It is well-known that the Fibonacci lattice has order optimal $L_\infty$- and $L_2$-discrepancy 
For the $L_\infty$-discrepancy we refer to the monograph \cite{N92} of H. Niederreiter.
For the classical $L_2$-discrepancy this was first proved by V. S\'os and S. K. Zaremba in \cite{SZ79}. 
For the periodic $L_2$-discrepancy it is even conjectured that the Fibonacci lattice is globally optimal among all point sets with the same number of points, see \cite{HO2016}. This is proved in \cite{HO2016} for $n = F_m \le 13$. Among integration lattices, the Fibonacci lattice has minimal periodic $L_2$-discrepancy at least if $n=F_m \le 832040$. This can be shown by a not particularly sophisticated exhaustive search through all integration lattices using a suitable simplification of the Warnock formula for the periodic $L_2$-discrepancy of integration lattices.

For the dispersion, it was proved by V. Temlyakov in \cite{T19} that the Fibonacci lattice is order optimal, i.e. that there exists a constant $c$ such that 
$$ \disp(\fiblat{m}) \le \frac{c}{F_m}. $$
The main purpose of this note is to show that the bound in Theorem \ref{thm:lowerbound} is actually sharp for the Fibonacci lattices.
In particular, we show the following theorem.

\begin{theorem}\label{thm:upperboundfib}
 Let $m\ge 3$ be an integer. The Fibonacci lattice $\fiblat{m}$ satisfies
 $$ \disp(\fiblat{m}) = \frac{2}{F_m}. $$
\end{theorem}

It may be conjectured that, up to torus symmetries, the Fibonacci lattices are the only point sets meeting the lower bound in Theorem \ref{thm:lowerbound}. 
This is not true.
The second purpose of this note is to discuss the structure of general optimal sets.
At least for integration lattices, we get a complete characterization. 

\begin{theorem}\label{thm:intlats}
 If $ \disp(\pointset{n}) = 2/n $ for some integration lattice $\pointset{n}$ with $n \ge 2$ points, then $n=F_m$ is a Fibonacci number and $\pointset{n}$ is torus symmetric to the Fibonacci lattice $\fiblat{m}$ or $n=2 F_m$ is twice a Fibonacci number and $\pointset{n}$ is torus symmetric to the lattice with generator $q=2 F_{m-2}$.
\end{theorem}

Clearly, Theorem \ref{thm:intlats} immediately implies Theorem \ref{thm:upperboundfib}, so we will prove only Theorem \ref{thm:intlats}.

For the convenience of the reader, we provide a short proof of Theorem \ref{thm:lowerbound} in 
Section \ref{sec2}. 
In particular, this proof also shows that any point set $\pointset{n}$ with $n$ points satisfying $ \disp(\pointset{n}) = 2/n$ 
has to be of a certain structure. This structure is then further employed in Section \ref{sec4a} to give examples of point sets with optimal dispersion that are not integration lattices. 

Section \ref{sec3} contains the proof of our main result Theorem \ref{thm:upperboundfib}.  
In Section \ref{sec4} we compute the nonperiodic dispersion of the Fibonacci lattice, which turns out to be only slighty smaller than $2/F_m$. 
We finish with a final section containing a discussion of related results in particular also in higher dimensions. 

% % % % % % % % % % % % % % % % % % % % % %
\section{Proof of Theorem \ref{thm:lowerbound}} \label{sec2}
% % % % % % % % % % % % % % % % % % % % % % 

We now give a simple proof of  Theorem \ref{thm:lowerbound}. This proof is basically the same as the proof for general dimension in \cite{U2018}.

Fix a point set $\pointset{n}$ with $n$ points.
For $x\in [0,1)$, let $n(x)$ be the number of points in the (periodic) rectangle $B(x)=[x,x+2/n) \times [0,1)$. Then each point in  $\pointset{n}$ is in $B(x)$ for a set of $x$ of measure exactly $2/n$. Hence
$ \int_0^1 n(x) \dint x = 2.$ 

Assume first that $n(x)<2$ for some $x$. Then, for this $x$, either $n(x)=0$ or $n(x)=1$. 
Then, for some $\eps>0$, also the rectangle $B=[x-\eps,x+2/n) \times [0,1)$ contains at most one point of $\pointset{n}$. Splitting the box along the second coordinate of this point, if it exists, we obtain a periodic rectangle of size  $2/n+\eps$ containing no points of 
$\pointset{n}$ in its interior showing that $ \disp(\pointset{n})>2/n$.

If $n(x) > 2$ for some $x$, then $ \int_0^1 n(x) \dint x = 2$ implies that $n(x)<2$ for some $x$, again $ \disp(\pointset{n})>2/n$ follows.

The only case not considered is the case that $n(x)=2/n$ for every $x\in [0,1)$. 
Let $x$ be the first coordinate of a point in the point set. Then we obtain that there exists exactly one point in the pointset with $x$-coordinate in $(x,x+2/n)$ and that there is exactly one point in the pointset with $x$-coordinate equal to $x+2/n$. In particular, splitting the rectangle $(x,x+2/n) \times [0,1]$ along the second coordinate of the (only) point in this rectangle gives an empty rectangle of size $2/n$ and $ \disp(\pointset{n}) \ge 2/n$ follows. 

Moreover, if $n$ is odd,
this implies that the $x$-coordinates of  points of $\pointset{n}$ form the set $\{ \xi+k/n : k=0,1,\dots,n-1\}$ for some $\xi\in [0,1/n)$.  
If $n$ is even, the situation is a little different. Then $n(x)=2/n$ for every $x\in [0,1)$ only implies that  $\pointset{n}$ is the union of two sets 
$\{ \xi_i+2k/n : k=0,1,\dots,n/2-1\}$ for some $\xi_1,\xi_2 \in [0,2/n)$.
Similar reasoning can be applied to the second coordinate instead of the first coordinate.

Altogether, we proved Theorem \ref{thm:lowerbound} together with structural properties of pointsets meeting the bound. In particular, if $n$ is odd,  any pointset $\pointset{n}$ with $n$ points satisfying $ \disp(\pointset{n}) = 2/n$ is, up to torus symmetries, a lattice point set  of the type
$$\Big\{\Big(\frac{k}{n},\Big\{\frac{\pi(k)}{n}\Big\}\Big): k\in\{0,1,...,n-1\}\Big\}$$
for some permutation $\pi$ of the set $\{0,1,...,n-1\}$.

% % % % % % % % % % % % % % % % % % % % % %
\section{Proof of Theorems \ref{thm:upperboundfib} and \ref{thm:intlats}} \label{sec3}
% % % % % % % % % % % % % % % % % % % % % %

Throughout this section we fix an integration lattice 
$$\mathcal{P}_n = \Big\{\Big(\frac{k}{n},\Big\{\frac{k q}{n}\Big\}\Big):k\in\{0,1,...,n-1\}\Big\}.$$
containing $n$ points with generator $q\in \{1,2,\dots,n-1\}$. 
We will prove Theorem \ref{thm:intlats}, Theorem \ref{thm:upperboundfib} is a direct consequence.
Our proof of Theorem \ref{thm:intlats} relies on a careful examination of the length of the intervals obtained by splitting the torus with the points $\big(\big\{\frac{k q}{n}\big\}\big)$.
To simplify the notation, we will scale the one-dimensional torus by a factor of $n$ and consider the sequence $\big( n\big\{\frac{kq}{n}\big\}\big).$ To this end, let
$y:\{0,1,...,n-1\}\to\{0,1,...,n-1\}$ be the function defined as 
$$y(k):=n\Big\{\frac{kq}{n}\Big\}.$$ 
Furthermore, let $$Y_\ell:=\big(y(k)\big)_{k=0}^{\ell-1}$$ denote the sequence of the first $\ell$ function values of $y$.

We now want to consider the distances between consecutive elements of the sequence $Y_\ell$.

\begin{definition}\label{def:distance}
Let $(x_k)_{k=0}^{\ell-1}$ be a sequence of $\ell$ elements of the one-dimensional torus scaled by $n$. 
Let $(y_k)_{k=0}^{\ell-1}$ be the non-decreasing rearrangement of the sequence $(x_k)_{k=0}^{\ell-1}$.
For $a,b \in [0,n]$ with $a \neq b$, let $d(a,b)$ denote the oriented scaled torus distance of the points $a$ and $b$, i.e. $d(a,b)= b-a$ for $b<a$ and $d(a,b)=n+b-a$ if $b<a$. We also set $d(a,a) = n$. We say that $c\in(0,n]$ is a distance of the sequence $(x_k)$ if there exists an $i\in\{1,...,\ell-1\}$ such that $$d(y_{i-1},y_i) = c$$ or if $$ d(y_{\ell-1},y_0) = c.$$
\end{definition}

The following lemma is a direct consequence of the Three-distance or Three-gap Theorem conjectured by H. Steinhaus and proved in the late 1950s by V. S\'os \cite{S1958}, J. Sur\'anji \cite{S1958a}, and S. \'Swierczkowski \cite{S1959}.
\begin{lem}\label{lem:threedistances}
For any $\ell\in \{1,...,n\}$, the sequence $Y_\ell$ has at most three different distances. If $Y_\ell$ has three different distances $d_1>d_2>d_3$, then $d_1 = d_2 + d_3$.
\end{lem}

We will now investigate how often those three distances occur. The following definition will be helpful in simplifying the notation. We also refer to the figure below for an instructive example.

\begin{definition}\label{def:distancepartition}
Let the sequence $Y_\ell$ have the distances $d_1 > d_2 > d_3$ $a_1,a_2,a_3$ times, respectively. Then we say $Y_\ell$ induces the splitting $$n = a_1d_1+a_2d_2+a_3d_3.$$ Notice that the equality holds if we interpret it algebraically.
If there are only one or two distances, the notation is used accordingly. In that case, we also use the notation above and allow $d_3=d_2$ if $a_3=0$ and $d_2=d_1$ if $a_3=a_2=0$.
\end{definition}

\begin{figure}\label{fig}
 \caption{Consecutive Splittings for $n=13$ and $q=5$}
 \includegraphics[scale=0.62]{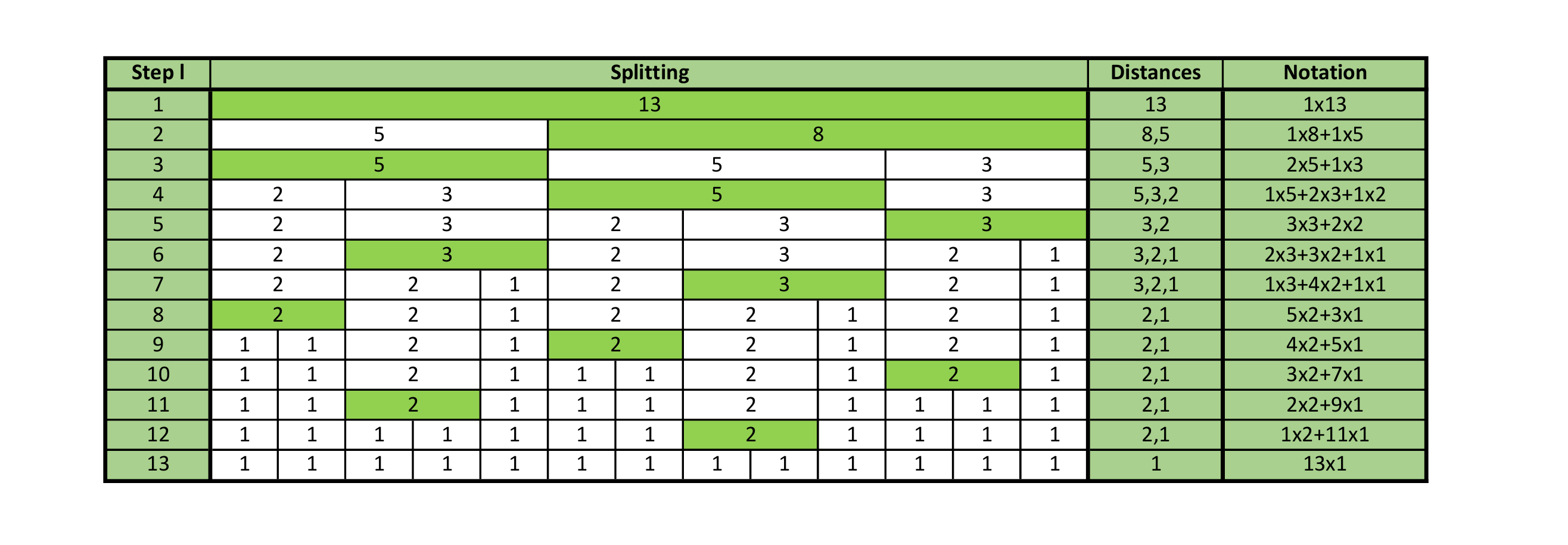}
\end{figure}

As it turns out, if we increase the number $\ell$ of points considered, we always end up splitting the largest distance.

\begin{lem}\label{lem:splitlargestdistance}
By going from $Y_\ell$ to $Y_{\ell+1}$, the largest distance from $Y_\ell$ is split.
\end{lem}

\begin{proof}
If $Y_\ell$ only has one distance, this is trivial. Suppose now that $Y_\ell$ induces the splitting $$n = a_1d_1+a_2d_2+a_3d_3,$$ where $a_3\ge 0$, i.e. we also consider the case that $Y_\ell$ only has two different distances. Clearly, $Y_n$ splits the torus into equidistant intervals, i.e. $Y_n$ has only one distance. Suppose now that by going from $Y_\ell$ to $Y_{\ell+1}$, we split either $d_2$ or $d_3$. Without loss of generality, we will assume we split $d_2$. By going from $Y_\ell$ to $Y_{\ell+1}$, we introduced the point $y(\ell)$. This point has a left neighbour, i.e. there exists an $a < \ell$ such that $B(y(a),y(\ell))$ contains no other point of the sequence $Y_{\ell+1}$. In the same way there exists a right neighbour $y(b)$. Since we split the distance $d_2$, we know that $d(y(a),y(b))=d_2$. It is easy to see that for any $k$ such that $\ell+k<n$ we have that $y(\ell+k)$ is in the (periodic) interval $(y(a+k),y(b+k)$. Thus, any point introduced after $y(\ell)$ can only split a distance which is at most $d_2$. This means that the distance $d_1$ is never split again. However, since $Y_n$ should induce a splitting with only one distance, this is clearly a contradiction. Thus, we always split $d_1$.
\end{proof}

From now on we will make additional assumptions on $n$ and $q$. On the one hand, we assume without loss of generality that $2 q \le n$. This is possible, since the integration lattices induced by $(n,q)$ and $(n,n-q)$ are torus symmetric. On the other hand, we assume that the integration lattice has optimal dispersion $2/n$, or, since we consider the scaled torus, dispersion $2n$. Since we are only interested in finding all optimal integration lattices, this is no real restriction.

The following lemma will give us an explicit formula for the splitting of $Y_\ell$ that will then directly imply Theorem \ref{thm:intlats}. Here we also use Fibonacci numbers $F_k$ with $k\le 0$, which satisfy the same recursion as for $k>0$.

\begin{lem}\label{lem:intlatticefibapproximation}
Let $m,k,j$ be positive integers satisfying $3 \le k \le m$, $F_m \le n < F_{m+1}$ and $1 \le j \le F_{k-2}$. 

If $k$ is odd, then
$Y_{F_k-j}$ induces the splitting 
\begin{equation} \label{eq:idsplittingoddk}
 n=j(F_{k-3}q-F_{k-5}n) + (F_{k-1}-j)(F_{k-4}n-F_{k-2}q) + (F_{k-2}-j)(F_{k-1}q-F_{k-3}n).
\end{equation}
Moreover, the fraction $\frac{n}{q}$ satisfies 
\begin{equation} \label{eq:splittingoddk}
 \frac{F_k}{F_{k-2}}\le \frac{n}{q}\le\frac{F_{k-1}}{F_{k-3}}.
\end{equation}

If $k$ is even, then $Y_{F_k-j}$ induces the splitting 
\begin{equation} \label{eq:idsplittingevenk}
n=j(F_{k-5}n-F_{k-3}q)+(F_{k-1}-j)(F_{k-2}q-F_{k-4}n)+(F_{k-2}-j)(F_{k-3}n-F_{k-1}q).
\end{equation}
Moreover, the fraction $\frac{n}{q}$ satisfies 
\begin{equation} \label{eq:splittingevenk} 
\frac{F_{k-1}}{F_{k-3}}\le\frac{n}{q}\le\frac{F_k}{F_{k-2}}.
\end{equation}

Moreover, the inequalities \eqref{eq:splittingoddk} and \eqref{eq:splittingevenk} for $3 \le k \le m$ are not only necessary but also sufficient for $\mathcal{P}_n$ to have minimal dispersion $2/n$.
\end{lem}

\begin{proof}
We will prove this Lemma by induction on $k$.

Let $k=3$. The only $j$ we need to consider is $j=1$. We need to examine the splitting $Y_{F_3-1} = Y_1$. $Y_1$ trivially splits the scaled torus into
\begin{eqnarray*}
n &=& 1(F_0 q-F_{-2}n)+(F_2-1)(F_{-1}n-F_1 q)+(F_1-1)(F_2 q-F_0n)\\
&=& 1n.
\end{eqnarray*}
Furthermore, since $2q\le n$, the distances $F_0q-F_{-2}n\ge F_{-1}n-F_1q\ge F_2q-F_0n$ are ordered. 
Also \eqref{eq:splittingoddk}, which reads as 
$$\frac{2}{1} = \frac{F_3}{F_1}\le\frac{n}{q}\le\frac{F_2}{F_0} = +\infty,$$ 
holds since $2q\le n$.

Let $k=4$. Again, the only $j$ we need to consider is $j=1$. 
We need to examine the splitting $Y_{F_4-1} = Y_2$. 
But $Y_2$ trivially splits the scaled torus into
\begin{eqnarray*}
n &=& 1(F_{-1}n-F_1q)+(F_3-1)(F_2q-F_0n)+(F_2-1)(F_1n-F_3q)\\
&=& 1(n-q) + 1q
\end{eqnarray*}
as claimed. 
We assumed that the integration lattice has optimal dispersion, i.e. there is no empty box of size greater than $2n$. The splitting $Y_2$ gives us a box of size $3(n-q)$. 
Thus, we get 
$$3(n-q)\le 2n \qquad \Longleftrightarrow \qquad \frac{n}{q}\leq 3.$$ 
Together with the bound from the case $k=3$, this implies 
$$\frac{2}{1} = \frac{F_3}{F_1} \le\frac{n}{q}\le \frac{F_4}{F_2} = \frac{3}{1},$$ 
which is \eqref{eq:splittingoddk}. 
Moreover, it again follows that the distances of the splitting $F_{-1}n-F_1q\ge F_2p-F_0n\ge F_1n-F_3q$ are ordered.

We will now assume the Lemma has been proven for $k,k+1$ with $k$ odd and we will prove it for $k+2$ and $k+3$. Of course, we then have to assume $m\ge k+2$ and  $m\ge k+3$, respectively.

We start with the proof for $k+2$. 
We need to consider the splitting 
$Y_{F_{k+2}-j}$ for $j\in\{1,...,F_k\}$. 
This splitting still exists, since $m\ge k+2$. 
We know that $Y_{F_{k+1}-1}$ gave us the splitting $$n=1(F_{k-4}n-F_{k-2}q)+(F_k-1)(F_{k-1}q-F_{k-3}n)+(F_{k-1}-1)(F_{k-2}n-F_kq).$$ 
Since, by Lemma \ref{lem:splitlargestdistance}, we always split the largest distance and the distances in the splitting above are ordered, 
$Y_{F_{k+1}}$ gives us 
$$n=F_k(F_{k-1}q-F_{k-3}n)+F_{k-1}(F_{k-2}n-F_kq).$$ 
Now we need to split the largest distance $F_{k+2}-j-F_{k+1}$ more times. Thus, $Y_{F_{k+2}-j}$ gives us $$n=j(F_{k-1}q-F_{k-3}n) + (F_{k+1}-j)(F_{k-2}n-F_kq) + (F_k-j)(F_{k+1}q-F_{k-1}n).$$ 
This shows \eqref{eq:idsplittingevenk}. 
It remains to check that the distances are ordered.
To this end, we use the minimality with respect to the dispersion. 
The splitting gives us an empty box of size 
$(F_{k+2}-j+1)(F_{k-1}q-F_{k-3}n)$. 
From the assumption on the dispersion we conclude that 
$(F_{k+2}-j+1)(F_{k-1}q-F_{k-3}n)\le 2n$. 
Of course, if this condition is satisfied for $j=1$, it is satisfied for any $j\in\{1,...,F_k\}$. 
Thus, we have 
$$F_{k+2}(F_{k-1}q-F_{k-3}n)\le 2n \qquad \Longleftrightarrow \qquad \frac{F_{k+2}}{F_k}=\frac{F_{k+2}F_{k-1}}{F_{k+2}F_{k-3}+2}\le\frac{n}{q}.$$ 
Together with the bounds \eqref{eq:splittingevenk} with $k+1$ instead of $k$, 
we get the new bounds 
$$\frac{F_{k+2}}{F_k}\le\frac{n}{q}\le\frac{F_{k+1}}{F_{k-1}}.$$ 
These are the bounds \eqref{eq:splittingoddk}
with $k+2$ instead of $k$, which now also imply that the distances of the splitting were ordered. 
The proof for $k+3$ is completely analogous.
\end{proof}

\begin{proof}[Proof of Theorem \ref{thm:intlats}]
%If $m=2$ we have $F_2 = 1$ point. Because of torus symmetries, it does not matter where we put that point, so our Fibonacci lattice is optimal and every other point set is torus symmetric to the Fibonacci lattice.
Assume that the integration lattice $\pointset{n}$ satisfies $\disp(\pointset{n})=2/n$.
Let $q$ be the generator of $\pointset{n}$ and assume that $2 q \le n$, passing to a torus equivalent integration lattice if necessary.
Let the positive integer $m$ be such that $n\in\{F_m,...,F_{m+1}-1\}$.
Since $n\ge 2$, we have $m\ge 3$.
If $m$ is odd, Lemma \ref{lem:intlatticefibapproximation} gives us for $k=m$ that $\frac{n}{q}$ must satisfy the inequalities 
$$\frac{F_m}{F_{m-2}}\le\frac{n}{q}\le\frac{F_{m-1}}{F_{m-3}}.$$ 
 Since fractions of Fibonacci numbers are optimal rational approximations of $\varphi$ (or in that case $\varphi^2$), the next rational approximation better than $\frac{F_m}{F_{m-2}}$ and $\frac{F_{m-1}}{F_{m-3}}$ would be $\frac{F_{m+1}}{F_{m-1}}$. 
But this is not possible, since $n<F_{m+1}$. Thus, if $\frac{n}{q}$ satisfies the above inequality, it has to be equal to one of the two sides.

If $\frac{n}{q} = \frac{F_m}{F_{m-2}}$, then $n=F_m$ and $q=F_{m-2}$ because of the restrictions on $n$ (since $F_{m+1}-1 < 2F_m$) and we have that $\pointset{n} = \fiblat{m}$ is a Fibonacci lattice.

If $\frac{n}{q} = \frac{F_{m-1}}{F_{m-3}}$, then $n=2 F_{m-1}$ because of the restrictions on $n$.
This implies $q=2 F_{m-3}$.

We conclude that the only possible optimal integration lattices are the lattices  described in the theorem.
Moreover, since the inequalities \eqref{eq:splittingoddk} and \eqref{eq:splittingevenk} for $3 \le k \le m$ are sufficient for $\mathcal{P}_n$ to have minimal dispersion $2/n$, these lattices indeed have optimal dispersion $2/n$.
\end{proof}

\begin{remark}\label{rem:formula}
 The proof of optimality of the Fibonacci lattice without characterizing all optimal integration lattices can be significantly simplified. In fact, it is then easier to directly show the formula
$$
 \disp(\fiblat{m}) = \frac{1}{F_m^2} \max_{3\le k \le m} F_k F_{m-k+3},
$$
which also follows from the above argument.
The maximum is attained for $k=3$ and $k=m$.
This was independently observed by M. Ullrich.
\end{remark}

\section{Optimal sets that are not integration lattices} \label{sec4a}
% % % % % % % % % % % % % % % % % % % % % % 

In this section, we give examples of pointsets $\pointset{n}$ satisfying  $\disp(\pointset{n}) =2/n$ that are not integration lattices. Of course, the restrictions given in Section \ref{sec2} have to be satisfied. These examples are obtained from the Fibonacci lattices $\fiblat{m}$ for even $F_m$ by shifting every other point by a fixed small vector, see Figure \ref{fig:distfib}. This leads to the distorted Fibonacci lattices 
\begin{align*}
\fiblat{m,\xi,\eta} := 
& \Big\{\Big(\frac{k}{F_m},\Big\{\frac{kF_{m-2}}{F_m}\Big\}\Big): k\in\{0,2,...,F_m-2\}\Big\} \\ 
& \, \cup \, \Big\{\Big(\frac{k}{F_m} + \frac{\xi}{F_m} ,\Big\{\frac{kF_{m-2}}{F_m}\Big\} + \frac{\eta}{F_m} \Big): k\in\{1,3,...,F_m-1\}\Big\}
\end{align*}
with $0 \le \xi,\eta < 1$. It turns out that for small enough $\xi$ and $\eta$, such a distortion does not alter the dispersion of the Fibonacci lattice.

For simplicity, we just analyse the case $\eta=0$ more closely, i.e., half of the Fibonacci lattice is shifted in the direction of the first coordinate. Then the argument from the previous section or the more direct argument mentioned in Remark \ref{rem:formula} lead to the conclusion that, 
as long as 
$$F_3F_m\ge\max_{4\le k\le m-1} (F_k+\xi)F_{m-k+3},$$
the dispersion does not grow. The maximum is attained (at least asymptotically) for $k=4$ and $k=5$. So we get the condition 
$$F_3F_m\ge (F_4+\xi)F_{m-1},$$
which is asymptotically equivalent to 
$$\xi\le \lim_{n\to \infty}\frac{F_3F_m}{F_{m-1}}-F_4 = 2\phi-3 = 0.236068\dots.$$

\begin{figure}[htb]\label{fig:distfib}
\caption{Fibonacci lattice, usual and distorted}
    \centering
    \begin{minipage}[t]{0.5\linewidth}
        \includegraphics[width=0.95\linewidth]{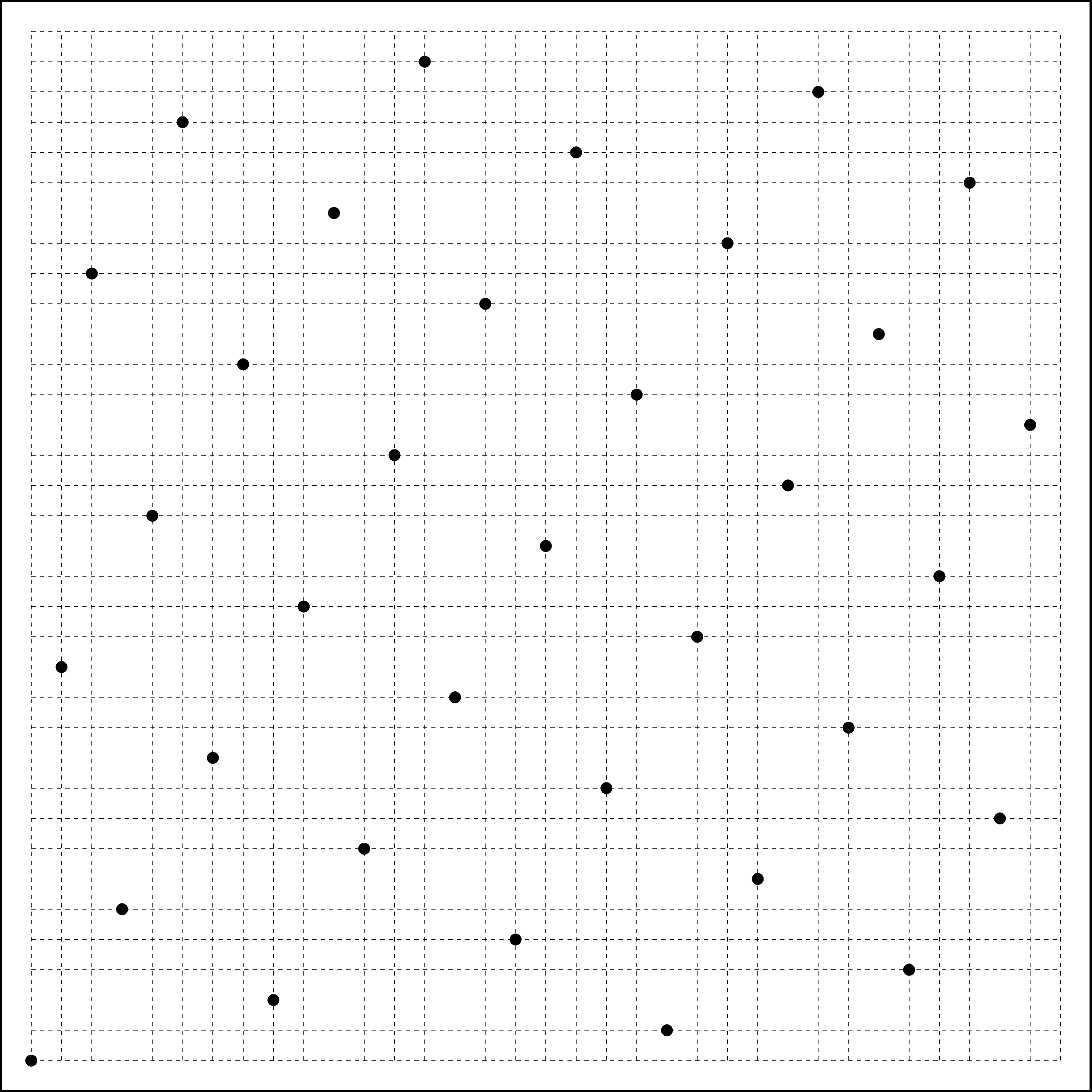}        
    \end{minipage}% 
    \hfill
    \begin{minipage}[t]{0.5\linewidth}
        \includegraphics[width=0.95\linewidth]{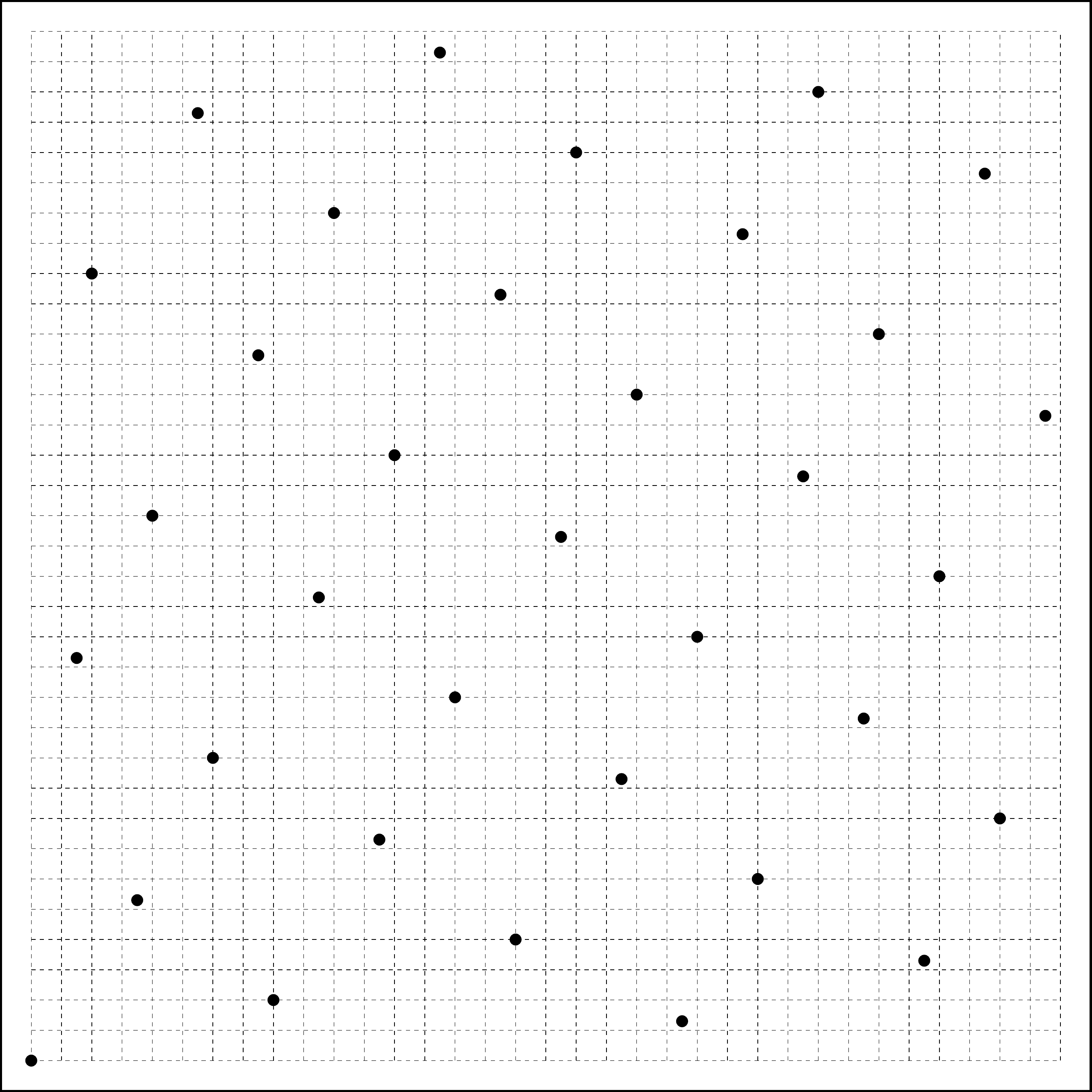}
    \end{minipage}% 
\end{figure}

The distorted Fibonacci lattices above are neither integration lattices nor lattice point sets. We could not decide if there are lattice point sets with large cardinality that both have optimal dispersion and are not torus equivalent to an integration lattice. 

% % % % % % % % % % % % % % % % % % % % % %
\section{The nonperiodic case} \label{sec4}
% % % % % % % % % % % % % % % % % % % % % % 

In this section, we study the dispersion of the Fibonacci lattice in the non-periodic case. Basically this means that we restrict the allowed rectangles in the definition of the dispersion to rectangles $B(x,y)$ where $x\le y$ coordinatewise. Let $\disp^\ast(\pointset{n})$ denote the corresponding dispersion of a point set $\pointset{n} \subset [0,1]^2$.

The best known lower bound (for $n\ge 16$) 
\begin{equation}
\label{eq:dj}
 \disp^\ast(\pointset{n}) \ge \frac{5}{4(n+5)}
\end{equation}
was proved in \cite{DJ2013}.
As far as we know, until now, the best known upper bound in dimension 2 for large $n$ is $\disp(\pointset{n}) \le 4/n$ if $n=2^m$ for some positive integer $m$ and a $(0,m,2)$-net $\pointset{n}$ in base 2. For general $n$, this implies that there exist pointsets $\pointset{n}$ of cardinality $n$ with $\disp(\pointset{n}) \le 8/n$.
Already the periodic dispersion of the Fibonacci lattices allows an improvement of these upper bounds.
We compute the nonperiodic dispersion of the Fibonacci lattice to further improve these bounds. 

\begin{theorem}\label{thm:upperboundfibnonperiodic}
 Let $m\ge 6$ be an integer. The Fibonacci lattice $\fiblat{m}$ without the point $(0,0)$ satisfies
 $$ \disp^\ast(\fiblat{m} \setminus \{(0,0)\}) = \frac{ 2(F_{m}-1)}{F_m^2}. $$
\end{theorem}

We only sketch the proof here. The proof for the periodic case in particular shows that the maximal periodic boxes containing no points of $\fiblat{m}$ in the interior have sidelength $F_j/F_m$ and $F_{m+3-j}/F_m$ for some $j=3,4,\dots,m$ leading to the formula
$$
 \disp(\fiblat{m}) = \frac{1}{F_m^2} {\max \left\{ F_j F_{m+3-j} \,:\, j=3,4,\dots,m\right\}} .
$$
The maximum is attained for $j=3$ and $j=m$. But the corresponding rectangles with sidelength $2/F_m$ and $F_m/F_m=1$ are true periodic rectangles wrapping around one direction. However, it is easy to see that there are still nonperiodic rectangles with sidelength $2/F_m$ and $(F_m-1)/F_m$. Those rectangles have an area of $2(F_m-1)/F_m^2.$

On the other hand, for each $j=4,5\dots,m-1$, there are nonperiodic rectangles with sidelength $F_j/F_m$ and $F_{m+3-j}/F_m$  that do not contain any point of $\fiblat{m}$ in the interior. In the nonperiodic setting, the point $(0,0)$ can be safely omitted. We arrive at
$$
 \disp^\ast(\fiblat{m} \setminus \{(0,0)\}) = \frac{1}{F_m^2} \max \big\{ 2(F_m-1), \max \left\{ F_j F_{m+3-j} \,:\, j=4,5,\dots,m-1\right\}\big\}
$$
for $m\ge 5$. It is not too hard to check that, for $m\ge 6$, this maximum is attained for $j=5$. Clearly, $$\frac{F_5F_{m-2}}{F_m^2}\le \frac{2(F_m-1)}{F_m^2}$$ for $m\ge 6$. This leads to the claim of the theorem.

% % % % % % % % % % % % % % % % % % % % % %
\section{Further results, final remarks and open problems} \label{sec5}
% % % % % % % % % % % % % % % % % % % % % % 

To put our results in the two-dimensional case into perspective, we now also consider the general $d$-dimensional case. Let $\disp(n,d)$ and $\disp^*(n,d)$ be the minimal dispersion of all pointsets $\pointset{n}$  in $[0,1]^d$ of cardinality $n$ in the periodic and non-periodic setting, respectively. The modifications in the definitions should be obvious.

The known lower and upper bounds imply that 
$$
 0 < a(d) := \liminf_{n\to\infty} n  \disp(n,d) \le \limsup_{n\to\infty} n  \disp(n,d) =: b(d) < \infty
$$
and 
$$
 0 < a^\ast(d) := \liminf_{n\to\infty} n \disp^\ast(n,d) \le \limsup_{n\to\infty} n  \disp^\ast(n,d) =: b^\ast(d) < \infty.
$$
The inequalities
\begin{equation}\label{eq:triv}
 a^\ast(d) \le a(d) \qquad \text{and} \qquad b^\ast(d) \le b(d)
\end{equation}
are trivial.
It is natural to study these quantities and to determine if $a(d)=b(d)$ and/or $a^\ast(d)=b^\ast(d)$, i.e., if the limits $\lim_{n\to\infty} n  \disp(n,d)$ and/or $\lim_{n\to\infty} n  \disp^\ast(n,d)$ exist. 
For $d=1$, equidistant points are optimal. This implies $a(1)=b(1)=a^\ast(1)=b^\ast(1)=1$. 

Already the case $d=2$ is much more difficult. In the periodic case, Theorems \ref{thm:lowerbound} and \ref{thm:upperboundfib} show that 
\begin{equation}\label{eq:a2b2}
 a(2) = 2 \qquad \text{and} \qquad 
 b(2) \le \frac{3+\sqrt{5}}{2} = 2.6180339\dots .
\end{equation}
Here $b(2)$ is estimated via monotonicity of $\disp (n,d)$ in $n$ together with $ \disp (n,d) = 2/n$ if $n$ is a Fibonacci number or twice a Fibonacci number.
The problem, if $b(2)=a(2)$ and, if not, the computation of $b(2)$ remain open. In the non-periodic case, the lower bound \eqref{eq:dj}, Theorem \ref{thm:upperboundfibnonperiodic} and inequalities \eqref{eq:triv} and \eqref{eq:a2b2} show that
$$
 \frac{5}{4} \le a^\ast(2) \le 2 \qquad \text{and} \qquad 
 b^\ast(2) \le \frac{3+\sqrt{5}}{2} = 2.6180339\dots .
$$
The exact determination of $a^\ast(2)$ and $b^\ast(2)$ remains open.

For general $d$, we only know that
$$ d \le a(d) \le 2^{7d} \qquad \text{and} \qquad b(d) \le 2^{7d+1} $$
as well as 
$$
  \frac{\log_2 d}{4} \le a^\ast(d)  \le 2^{7d} \qquad \text{and} \qquad b^\ast(d) \le 2^{7d+1}.
$$
The upper bounds follow from a construction using digital nets due to G. Larcher, see \cite{AHR17}.
The lower bound for $a(d)$ follows from the result of \cite{U2018}, the lower bound for $a^\ast(d)$ from the main result of \cite{AHR17}. Further upper bounds not  directly applicable to this problem or yielding worse bounds can be found in the papers \cite{K18,R18,T19}.

The lower bound from \cite{U2018} for general dimension $d$ is 
$$ \disp(\pointset{n}) \ge \frac{d}{n}, $$
which is equal to $d/n$ for $d\ge n$.
We now discuss the case of integration lattices $\pointset{n} \subset [0,1]^d$ with optimal periodic dispersion. It turns out that, in contrast to the already considered case $d=2$, such integration lattices can only exist for small $n$. We restrict the discussion here to the case $d=3$.

\begin{theorem}\label{thm:nointlats}
There are no integration lattices in $3$ dimensions which have more than $4$ points and satisfy 
$$ \disp(\pointset{n}) = \frac{3}{n}.$$
\end{theorem}

The crucial tool is the following additional information on the splittings induced by a two-dimensional integration lattice. Here we freely use the notation and language introduced in Section \ref{sec3}.

\begin{lem}\label{lem:intervalsnexttoeachother}
Let $\pointset{n} \subset [0,1]^2$ be an integration lattice with $n$ points and generator $q$. 
Assume that the induced splitting has more than one different distance (it has distances $d_1>d_2$ and maybe distance $d_3<d_2$). Then there is an empty interval of distance $d_1$ and an empty interval of distance at least $d_2$ which are next to each other.
\end{lem}

\begin{proof}
Assume we have an interval of distance $d_1$. Let the interval on the left and right have distance $d_3$. This distance $d_3$ must have come from splitting a distance $d_1$ or a distance $d_2$. A distance $d_2$ has not been split, as there is a distance $d_1$ remaining and we always split the largest distance. Thus, both distances of length $d_3$ next to the $d_1$-distance come from splitting a $d_1$-distance. A $d_1$ distance is always split into a $d_2$ and a $d_3$ distance. Furthermore, the $d_2$ is always to the left of the $d_3$ or the $d_2$ is always to the right of the $d_3$. In any way, if there were distances $d_1$ which were split to both sides of our $d_1$ distance, then there are now a $d_3$ and a $d_2$ distance next to our $d_1$.
\end{proof}

\begin{proof}[Proof of Theorem \ref{thm:nointlats}]
Observe that the projection of an integration lattice in dimension 3 onto any of the coordinate planes produces an integration lattice in dimension 2. Let $q$ be the  generator of one of those projected lattices. Splitting the 3-dimensional torus along the third coordinate of a point shows that a lower bound for the dispersion of the 3-dimensional integration lattice is given by the maximal size of a periodic rectangle for the 2-dimensional projection containing at most one point in the interior.

Now, assume that the first $k$ points o induce the splitting 
$$n=a_1d_1+a_2d_2+a_3d_3$$ 
for the 2-dimensional projected lattice.
Then, we observe the following: if $a_1>a_2+a_3$, an application of the pigeonhole  principle shows that there are two empty intervals of size $d_1$ next to each other. 
In any other case, Lemma
\ref{lem:intervalsnexttoeachother} tells us that there is an interval of size $d_1$ next to an interval of size $d_2$. Thus, in the first case, there is a two-dimensional box of size $(k+1)*(d_1+d_1)$ and in the second case there is a two-dimensional box of size $(k+1)*(d_1+d_2)$ which contains only one point. 
The sizes of those boxes are lower bounds for the $3$-dimensional dispersion.

Assume now that the 3-dimensional integration lattice has dispersion $3/n$. We will show a contradiction if $n$ is sufficiently large. Without loss of generality, assume $q\le n$. We consider now the projected 2-dimensional lattice with generator $q$. The first two points induce the splitting $$n = 1(n-q)+1q$$
and give rise to a relevant box  of size $$(2+1)\cdot (n-q+q) = 3n\le 3n.$$

The first $3$ points induce the splitting $$n=1(n-2q)+2q.$$
The proof will be finished by a giant case distinction:

\begin{enumerate}
\item $n-2q<q\Longleftrightarrow n<3q$. We have the largest box of size 
$$4(q+q)\le 3n\Longleftrightarrow\frac{8q}{3}\le n.$$ 
Together, we obtain $8q/3\le n<3q$. 
The next splitting is $$n=1q+2(n-2q)+1(3q-n).$$
We again need to distinguish two cases.
\begin{enumerate}
\item $n-2q<3q-n\Longleftrightarrow n<5q/2$. This is a contradiction.
\item $n-2q\ge 3q-n\Longleftrightarrow n\ge 5q/2$. We have the largest box $$5(q+n-2q)\le 3n\Longleftrightarrow n\le \frac{5q}{2}.$$ This is a contradiction.
\end{enumerate}
\item $n-2q\ge q \Longleftrightarrow n\ge 3q$. We have the largest box $$4(n-2q+q) \le 3n\Longleftrightarrow n\le 4q.$$ 
Together, we obtain $3q\le n\le 4q$. The next splitting is $$n=1(n-3q)+3q.$$
We again need to distinguish two cases.
\begin{enumerate}
\item $n-3q\ge q \Longleftrightarrow n\ge 4q$. Together with $n\le 4q$ we have $n=4q$. This is a contradiction for $n>4$.
\item $n-3q<q \Longleftrightarrow n<4q$. We have the largest box $$5(q+q)\le 3n\Leftrightarrow \frac{10q}{3}\le n.$$ 
Together, we obtain $10q/3\le n < 4q$. 
The next splitting is $$n=2q+2(n-3q)+1(4q-n).$$
We again need to distinguish two cases.
\begin{enumerate}
\item $n-3q\le 4q-n\Longleftrightarrow n\le 7q/2$. We have the largest box $$6(q+4q-n)\le 3n\Longleftrightarrow 10q/3\le n.$$ 
Together, we obtain $10q/3\le n\le 7q/2$. The next splitting is $$n=q+2(4q-n)+3(n-3q).$$
Luckily, we do not need a case distinction here, as we already know the ordering of the distances. We have the largest box $$7(q+4q-n)\le 3n\Longleftrightarrow 7q/2\le n.$$ Thus, $n=7q/2$. This is a contradiction for all $n$, except for $n=7$ and $q=2$. This case can be excluded  separately by exhaustively trying all possibilities. 
\item $n-3q>4q-n\Longleftrightarrow n>7q/2$. We have the largest box $$6(q+n-3q)\le 3n\Longleftrightarrow n\le 4q.$$ In total, $7q/2<n<4q$. The next splitting is $$n=q+3(n-3q)+2(4q-n).$$ Again, we do not need a case distinction. We have the largest box $$7(q+n-3q)\le 3n\Longleftrightarrow n\le 7q/2.$$ This is a contradiction.
\end{enumerate}
\end{enumerate}
\end{enumerate}
Every branch of the case distinction failed for $n>4$. This proves the theorem.
\end{proof}

\bibliographystyle{abbrv}

\end{document}